\documentclass[letterpaper,11pt]{amsart}

\usepackage[margin=1.4in]{geometry}
\usepackage{amscd, amssymb}
\usepackage{amsmath,amscd}
\usepackage{mathtools}
\usepackage[driverfallback=hypertex]{hyperref}
\usepackage{graphics}
\usepackage{graphicx}
\usepackage{color}
\usepackage{bm}
\usepackage{enumitem}
\setlist[enumerate,1]{label={(\alph*)}}
\setlist[enumerate,2]{label={(\roman*)}}

\newtheorem{thm}{Theorem}[section]
\newtheorem{conj}[thm]{Conjecture}
\newtheorem{pr}[thm]{Proposition}
\newtheorem{lem}[thm]{Lemma}
\newtheorem{lemma}[thm]{Lemma}

\theoremstyle{definition}
\newtheorem{definition}[thm]{Definition}

\newtheorem*{notation*}{Notation}

\theoremstyle{remark}

\newcommand{\ignore}[1]{}
\newcommand{\R}{\mathbb R}

\newcommand{\N}{\mathbb N}
\newcommand{\Prob}{\Pr}

\newcommand{\E}{{\mathbb{E}}}

\newcommand{\oone}{{o \left(1\right)}}
\newcommand{\omegaone}{{\omega \left(1\right)}}
\newcommand{\AP}[3]{{\tilde{\mathcal{M}} \left( #1, #2 , #3 \right)}}
\newcommand{\mnp}[4]{{ \mathcal{M} \left( #1, #2 , #3 ; #4 \right)}}
\newcommand{\mnmkp}{{\mnp{m}{n}{k}{p}}}
\newcommand{\geqone}[4]{{\mathcal{M} \left(#1, #2 , #3 ; #4; \geq 1 \right)}}

\newcommand{\geqnp}{{\geqone{n}{n}{m}{p}}}
\newcommand{\termdefine}[1]{\textit{#1}}
\newcommand{\mB}{{\mathcal{B}}}
\newcommand{\mC}{{\mathcal{C}}}
\newcommand{\mG}{{\mathcal{G}}}
\newcommand{\given}{{|}}

\DeclareMathOperator{\deg1}{deg}

\begin{document}
\title{On the threshold problem for Latin boxes}
\author{Zur Luria$^1$}
\thanks{$^1${\sc Israel Institute for Advanced Studies.} Email address: zluria@gmail.com}
\author{Michael Simkin$^2$}
\thanks{$^2${\sc Institute of Mathematics and Federmann Center for the Study of Rationality, The Hebrew University of Jerusalem, Israel.} Email address: menahem.simkin@mail.huji.ac.il}

\thanks{Part of this research was conducted at the Institute of Theoretical Studies, ETH, 8092 Zurich, Switzerland. Partially supported by Dr.~Max R\"ossler, the Walter Haefner Foundation and the ETH Foundation}

\begin{abstract}

Let $m \leq n \leq k$. An $m \times n \times k$ 0-1 array is a \termdefine{Latin box} if it contains exactly  $mn$ ones, and has at most one $1$ in each line. As a special case, Latin boxes in which $m = n = k$ are equivalent to Latin squares.

Let $\mathcal{M}(m,n,k;p)$ be the distribution on $m \times n \times k$ 0-1 arrays where each entry is $1$ with probability $p$, independently of the other entries. The threshold question for Latin squares asks when $\mathcal{M}(n,n,n;p)$ contains a Latin square with high probability. More generally, when does $\mathcal{M}(m,n,k;p)$ support a Latin box with high probability?

Let $\varepsilon>0$. We give an asymptotically tight answer to this question in the special cases where $n=k$ and $m \leq \left(1-\varepsilon \right) n$, and where $n=m$ and $k \geq \left(1+\varepsilon \right) n$. In both cases, the threshold probability is $\Theta \left( \log \left( n \right) / n \right)$. This implies threshold results for Latin rectangles and proper edge-colorings of $K_{n,n}$.

\end{abstract}

\maketitle

\pagestyle{plain}

\section{Introduction}

An order-$n$ \termdefine{Latin square} is equivalent to an $n \times n \times n$ 0-1 array with a single 1 in each line, where a line is the set of elements obtained by fixing the values of two indices and letting the third vary over $\left[n\right] \coloneqq \left\{1,...,n\right\}$. With this in mind, the following definition is natural.

\begin{definition}
Let $m \leq n \leq k$. An $m \times n \times k$ 0-1 array is a \termdefine{Latin box} if it contains exactly $mn$ ones, and at most one 1 in each line. 
\end{definition}

An $m \times n \times k$ Latin box is equivalent to a 3-uniform tripartite hypergraph on $m+n+k$ vertices such that each pair of vertices is contained in at most one edge, and the number of edges is maximal subject to this constraint. Thus, Latin boxes can be viewed as a $3$-uniform version of matchings of size $m$ in unbalanced bipartite graphs on $m+n$ vertices.

As additional motivation, consider the two following special cases. An $n \times n \times k$ Latin box $A$ is equivalent to a proper edge-coloring of the complete bipartite graph $K_{n,n}$ using $k$ colors. One obtains such a coloring from $A$ by coloring the edge $\{i,j\}$ with the unique index $c$ such that $A(i,j,c)=1$. The Latin box property ensures that this is a proper coloring. In addition, an $m \times n \times n$ Latin box $A$ is equivalent to an $m \times n$ Latin rectangle $R$ over the symbol set $[n]$, by setting $R(i,j)$ to be the index of the unique 1 in $A(i,j,\cdot)$. 

In this paper, we ask when a random three-dimensional 0-1 array contains a Latin box with high probability. Formally, let $\mnmkp$ be the distribution over $m \times n \times k$ 0-1 arrays where each element is 1 with probability $p$. A property of such an array is \termdefine{monotone} if changing zeros to ones cannot violate the property.

\begin{definition}
Let $m=m(n),k=k(n)$ satisfy $m \leq n \leq k$ for all $n \in \mathbb{N}$. $p_0=p(n)$ is a \termdefine{threshold} for a monotone property $\mathcal{P}$ if
\[
\lim_{n\rightarrow \infty}\Pr [M \sim \mnmkp \text{ satisfies $\mathcal{P}$}] = \begin{cases}
0 & \text{ if $p/p_0 \rightarrow 0$} \\
1 & \text{ if $p/p_0 \rightarrow \infty$}
\end{cases}.
\]
$p_0$ is a \termdefine{sharp threshold} for $\mathcal{P}$ if for every $\eta > 0$, 
\[
\lim_{n\rightarrow \infty}\Pr [M \sim \mnmkp \text{ satisfies $\mathcal{P}$}] = \begin{cases}
0 & \text{ if $p< (1-\eta) p_0$} \\
1 & \text{ if $p>(1+\eta) p_0$}
\end{cases}.
\]
\end{definition}

Our first result addresses the motivating case of $n=m=k$, namely, Latin squares. Here, and throughout the paper, we abuse notation and refer to $n \times n \times n$ Latin boxes as Latin squares.

\begin{thm} \label{thm:Latin}
There is an infinite family $F \subseteq \mathbb{N}$ and $p<1$ such that
\[
\lim_{n\in F,n \rightarrow \infty} \Pr [M \sim \mnp{n}{n}{n}{p} \text{ contains a Latin square}]=1.
\]
\end{thm}

This theorem is proved in Section \ref{sec:Latin}. It is actually an easy consequence of a stronger result of Andr{\'e}n, Casselgren, and \"Ohman \cite{andren2013avoiding}, who showed that an analogous minimum-degree result holds. We include it here because the proof is short and elegant. We also note that Keevash's method of randomized algebraic constructions \cite{Ke14,Ke15} can likely be used to show the existence of some $\varepsilon > 0$ for which $\mnp{n}{n}{n}{n^{-\varepsilon}}$ contains a Latin square with high probability. Showing this, however, is beyond the scope of this paper.

A recurring theme in the study of threshold properties is that an obvious obstruction for a property is essentially the \textit{only} obstruction for that property. For example, in the $G(n;p)$ model, a random graph contains a perfect matching w.h.p.\ whenever it contains no isolated vertices. In the case of Latin squares, the obvious obstruction is a line with no $1$s, corresponding to a threshold of $p=\log(n)/n$. This leads us to the following conjecture. 

\begin{conj}
The threshold for $M \sim \mnp{n}{n}{n}{p}$ to contain a Latin square is $p=\log(n)/n$.
\end{conj}

A similar conjecture was proposed by Casselgren and H{\"a}ggkvist \cite[Conjecture 1.4]{casselgren2016coloring}, although the underlying probability models are different.

The next theorem deals with the case $m < n = k$. It can be interpreted as a result on Latin rectangles. Following a common abuse of notation, here and in the rest of the paper we round large reals to the nearest integer. By an argument of van-Lint and Wilson \cite[Theorem 17.3]{van2001course}, the number of Latin squares is $((1 + o(1))n/e^2)^{n^2}$. Essentially the same argument implies that for fixed $\varepsilon \in (0,1)$, the number of $(1-\varepsilon)n \times n$ Latin rectangles is asymptotically $\left( \left( 1 + \oone \right) \left( \frac{1}{\varepsilon} \right)^{\varepsilon/(1-\varepsilon)} \frac{n}{e^2} \right)^{(1-\varepsilon)n^2}$. For the sake of completeness, we prove this assertion in Appendix \ref{ap:rectangle count}.

\begin{thm}\label{thm:latin_rectangle}
	Let $\varepsilon>0$. The threshold for $M \sim \mnp{(1-\varepsilon)n}{n}{n}{p}$ to contain a Latin box is $\log(n) / n$. Furthermore, if $p = \omega \left( \log (n) / n \right)$, then with high probability $M$ supports
$
\left( \left( 1 \pm \oone \right) \left( \frac{1}{\varepsilon} \right)^{\varepsilon/(1-\varepsilon)} \frac{n}{e^2} p \right)^{(1-\varepsilon)n^2}
$
Latin boxes.
\end{thm}

We prove this theorem in Section \ref{sec:rectangle}.
A recent work by Casselgren and H{\"a}ggkvist \cite{casselgren2016coloring} proved a similar result for ${1-o(n^{-1/2})<\varepsilon<1}$. Our theorem can be viewed as a strengthening of their result to any constant $\varepsilon>0$. 

The next theorem can be interpreted as a result on edge-coloring $K_{n,n}$ with $\left( 1 + \varepsilon \right) n$ colors. It is proved in Section \ref{sec:boxes}.

\begin{thm} \label{thm:list_coloring}
	Let $\varepsilon>0$. The threshold for $M \sim \mnp{n}{n}{(1+\varepsilon)n}{p}$ to contain a Latin box is $p = \frac{2 \log n}{(1+\varepsilon)n}$, and this threshold is sharp.
\end{thm}
In fact, we prove a stronger result (Theorem \ref{thm:hitting_time}): In the random process where, starting with the all zeros array, at each step we flip a randomly chosen 0 to 1, then with high probability the first time at which the array contains a Latin box is equal to the time at which every line of the form $\left( r,c , \cdot \right)$ contains at least one 1.

\subsection{Notation}

We use asymptotic notation in the usual way. For example, if $f,g:{\N\to(0,\infty)}$, then $f(n) = O \left( g(n) \right)$ means that $\limsup_{n\to\infty} f(n) / g(n) < \infty$. We also make use of asymptotic notation in arithmetic expressions. For example, by $f(n) = n + e^{O \left( g(n) \right)}$ we mean that there exists a function $h$ satisfying $h(n) = O \left( g(n) \right)$ and $f(n) = n + e^{h(n)}$.

\section{Proof of Theorem \ref{thm:Latin}} \label{sec:Latin}

\begin{proof}[Proof of Theorem \ref{thm:Latin}]

For $p \in (0,1)$, we define $F = \{2^k:k \in \mathbb{N} \}$, and give a recursive bound on 
\[
p_k = \Pr \left[ M \sim \mnp{2^k}{2^k}{2^k}{p} \text{ contains a Latin square} \right] .
\]
Consider first the case $k=1$. The probability that $M \sim \mnp{2}{2}{2}{p}$ contains a given order-$2$ Latin square is $p^4$. As there are exactly two such Latin squares, and they are disjoint, by the inclusion-exclusion principle the probability that $M$ contains a Latin square is $q(p) \coloneqq 2p^4-p^8$. 

For $k>1$, we view $M \sim \mnp{2^k}{2^k}{2^k}{p}$ as a $2 \times 2 \times 2$ block array, where each block is distributed according to $\mnp{2^{k-1}}{2^{k-1}}{2^{k-1}}{p}$. If there is an order-$2$ Latin square $L$ such that the blocks in $M$ corresponding to the $1$s of $L$ all contain order-$2^{k-1}$ Latin squares, then the union of these squares is a Latin square contained in $M$.

The probability that this happens is $q(p_{k-1})$, and so we have $p_k \geq q(p_{k-1})$ and $p_1=q(p)$. We note that the equation $q(x) = x$ has a unique solution $p^* \in (0,1)$, and that for $x \in \left( p^* , 1 \right)$, $q(x) > x$. Therefore, if $p>p^* \approx 0.9206$, the sequence $\{p_k\}_{k=1}^\infty$ is monotone increasing and bounded and hence convergent. Let $p' = \lim_{k\to\infty}p_k$. As $q$ is continuous and increasing on $(p^*,1]$ we have $p' \leq q \left( p' \right) = \lim_{k\to\infty} q(p_k) \leq \lim_{k\to\infty} p_{k+1} = p'$. The unique fixed point of $q$ in the interval $(p^*,1]$ is $1$, and so $p' = 1$.

\end{proof}

In order to obtain a better bound on $p$, 
one can in principle repeat the above argument for any fixed $n_0$. The probability that $M \sim \mnp{n_0}{n_0}{n_0}{p}$ contains a Latin square is given by some polynomial $q_{n_0}(p)$. One can compute $q_{n_0}(p)$ by listing all order-$n_0$ Latin squares and applying the inclusion-exclusion principle to calculate the probability that $M$ contains one of them. 

It is possible to show that there exists some $p_{n_0}^* \in (0,1)$, such that for $p$ between $p_{n_0}^*$ and $1$, $q_{n_0}(p)>p$. Indeed, fix two disjoint order-$n_0$ Latin squares, and let $\tilde{q}(p) = 2 p^{n_0^2}-p^{2 n_0^2}$ be the probability that $M \sim \mathcal{M}(n_0,n_0,n_0;p)$ contains at least one of them.
Clearly, $q_{n_0}(p) \geq \tilde{q}(p)$, and when $1-1/(2n_0^4)<p<1$ one can check that $\tilde{q}(p)>p$.

Set $F = \{n_0^k : k \in \N \}$.
Consider an $n_0^k \times n_0^k \times n_0^k$ 0-1 array $A$ as an $n_0 \times n_0 \times n_0$ block array consisting of $n_0^3$ blocks, each of which is an $n_0^{k-1}\times n_0^{k-1} \times n_0^{k-1}$ 0-1 array. 
We say that $A$ is a \emph{block} Latin square if there is some order-$n_0$ Latin square $L$ such that each block of $A$ is an order-$n_0^{k-1}$ Latin square if the corresponding element of $L$ is 1, or the all zero array otherwise.

Now, the probability $p_k$ that $M \sim \mnp{n_0}{n_0}{n_0}{p}$ contains a Latin square is bounded below by the probability that it contains a block Latin square, which is $q_{n_0}(p_{k-1})$. Therefore, if $p>p^*_{n_0}$, we have $\lim_{k \rightarrow \infty} p_k = 1$. For example, performing this calculation for $n_0 = 3$ gives $p_3^*\approx 0.86$. As a practical matter, however, this procedure seems computationally infeasible for much larger values of $n_0$.

\section{Proof of Theorem \ref{thm:latin_rectangle}} \label{sec:rectangle}

It is easy to show that for small $p$, with high probability $M \sim \mnp{(1-\varepsilon)n}{n}{n}{p}$ has an empty line of the form $M(i,j,\cdot)$. Indeed, the number of such lines is distributed binomially with parameters $(1 - \varepsilon )n^2, (1-p)^n$, and when $p<\frac{1}{2} \log(n)/n$, a second moment argument shows that with high probability there is such an empty line. In this case $M$ does not contain a Latin box.

For the upper bound, we show that for every $\varepsilon > 0$, there is a constant $C > 0$ depending only on $\varepsilon$ such that if $p \geq C \log (n) / n$, then w.h.p.\ $M \sim \mnp{(1-\varepsilon)n}{n}{n}{p}$ contains a Latin box. We present a randomized algorithm for finding a Latin box, and show that with high probability it succeeds.

Note that a Latin box in $M$ is a sequence of $\left(1-\varepsilon \right)n$ disjoint permutation matrices $P_i$, one in each plane of the form $M_i := M \left(i,\cdot,\cdot \right)$. Therefore, a natural algorithm to consider is to deal with these plane one by one, at each step choosing a permutation matrix supported by $M_i$ that does not conflict with previous choices.

To analyze this algorithm, consider the $i$-th step. At this stage, $(i-1)$ permutation matrices have already been chosen, ruling out exactly $(i-1)$ entries in each row and column of $M_i$. Our task is to find a permutation matrix supported by the remaining elements of $M_i$. 

Any $n\times n$ 0-1 matrix is the biadjacency matrix of a bipartite graph on $n + n$ vertices. In this language, the elements that have not been ruled out correspond to a regular bipartite graph $G_i$, and we want to find a perfect matching of a random subgraph of $G_i$, in which we keep each edge with probability $p$.

It is well known that with high probability a random bipartite graph has a perfect matching when it has no isolated vertices, which happens around $p=\log(n)/n$. The same holds for a random subgraph of a dense regular bipartite graph: Goel, Kapralov, and Khanna \cite[Theorem 2.1]{goel2010perfect} proved that there exists a constant $C>0$ s.t.\ if $G$ is a $k$-regular bipartite graph on $2n$ vertices, then a random subgraph of $G$ in which each edge is retained independently with probability $p = C n \log(n) / k^2$ contains a perfect matching with high probability. A careful analysis of their proof shows that if $C$ is large enough then the probability of failure is $o(1/n)$. In our context this implies that if $p \geq C \log(n) / \left( \varepsilon^2 n \right)$, then w.h.p. $M \sim \mnp{(1-\varepsilon)n}{n}{n}{p}$ contains a Latin box.

The arguments above determine the threshold for the appearance of Latin boxes in $\mnp{(1-\varepsilon)n}{n}{n}{p}$. In order to prove that w.h.p.\ $M$ contains close to the expected number of Latin boxes,
we modify the algorithm by requiring that each permutation matrix be chosen uniformly at random. As we will show, this ensures that with high probability the graphs $G_i$ are all pseudorandom.
We then prove that with high probability, a random subgraph of a sufficiently dense pseudorandom regular bipartite graph has many perfect matchings. 

Suppose that $f(n) = \omega(1)$ and $p = f(n)\frac{\log(n)}{n}$. Set $\delta = \max(f(n)^{-1/3},1/n) = o(1)$. Wherever necessary we assume that $n$ is sufficiently large for asymptotic inequalities to hold. Formally, at the $i$-th step we choose a permutation matrix uniformly at random from the set of permutation matrices supported by $M_i$ that are disjoint from previous choices. Now, set $k=k(i):=n-i+1$, and set $L=L(i):=(1-\delta)kp$. If the number of choices at step $i$ is less than $L^n \frac{n!}{n^n}$, the algorithm aborts.
We will show that with high probability, the algorithm does not abort, and therefore it succeeds in finding a Latin box. This implies the enumeration result.

Indeed, let $A$ be the number of Latin boxes supported by $M$. The probability of a specific Latin box being chosen by the algorithm is at most 
\[
Q = \prod_{i=1}^{(1-\varepsilon)n}{\left(L(i)^n \frac{n!}{n^n}\right)^{-1}}.
\]
Therefore, the probability that the algorithm succeeds is at most $AQ$. On the other hand, the algorithm succeeds w.h.p.\ and therefore, applying Stirling's approximation,
\[
A \geq (1-o(1))/Q = \left( \left( 1 \pm o(1) \right) \left( \frac{1}{\varepsilon} \right)^{\varepsilon/(1-\varepsilon)} \frac{n}{e^2} p \right)^{(1-\varepsilon)n^2}.
\]

The upper bound on $A$ follows from Markov's inequality, together with the observation that 
\[
\mathbb{E}[A] = \left( \left( 1 + o(1) \right) \left( \frac{1}{\varepsilon} \right)^{\varepsilon/(1-\varepsilon)} \frac{n}{e^2} p \right)^{(1-\varepsilon)n^2}.
\]

As described above, let $G_i$ be the $k$-regular bipartite graph corresponding to the elements that were not ruled out by previous choices. Let $H_i$ be the intersection of $G_i$ with the graph corresponding to $M_i$. Thus, $H_i$ is distributed as a random subgraph of $G_i$, where each edge is kept with probability $p$. It suffices to show that with high probability the graphs $H_i$ all have sufficiently many perfect matchings.

We say that a $k$-regular bipartite graph $G = \langle U \cup V,E\rangle$ is \termdefine{$c$-pseudorandom} if for every $X \subseteq U,Y \subseteq V$ such that $|X|,|Y| \geq \frac{\varepsilon}{10} n $, the number $E_G(X,Y)$ of edges between $X$ and $Y$ is at least $(1-c) |X||Y|\frac{k}{n}$. Our general strategy is to show that with high probability the graphs $G_i$ are all sufficiently pseudorandom, and that this implies the desired property for the graphs $H_i$. 

The following lemma will enable us to bound the number of perfect matchings in $H_i$ provided that $G_i$ is pseudorandom.

\begin{lemma}\label{lem:typicality}
	Let $G$ be a $k$-regular $\delta^{1/3}$-pseudorandom graph, and let $H$ be a random subgraph of $G$, in which each edge of $G$ survives with probability $p$. With probability $1 - n^{-\omegaone}$, the graph $H$ contains an $L$-factor, i.e.\ an $L$-regular spanning subgraph, where $L=(1-\delta)kp$.
\end{lemma}

The next lemma asserts that if the algorithm did not abort before the $i$-th step, then with high probability $G_i$ is $\delta^{1/3}$-pseudorandom. Its proof is reminiscent of the proof of \cite[Theorem 2]{kwan2016intercalates}.

\begin{lemma}\label{lem:Matthew's_method}
	Let $1 \leq i \leq (1 - \varepsilon)n$. Conditioned on the number of perfect matchings in $H_j$ being at least $L(j)^n \frac{n!}{n^n}$ for every $j<i$, the probability that $G_i$ is not $\delta^{1/3}$-pseudorandom is at most $\exp(-\Omega(n))$.
\end{lemma}

The Egorychev--Falikman theorem \cite{egorychev1981solution,falikman1981proof} states that the permanent of an order-$n$ doubly stochastic matrix is minimized by the matrix whose entries are all $1/n$, and is equal to $\frac{n!}{n^n}$. As the biadjacency matrix of an $L$-regular bipartite graph on $2n$ vertices is $L$ times a doubly stochastic matrix, this theorem implies that such a graph has at least  $L^n \frac{n!}{n^n}$ perfect matchings. In particular, if $H$ contains an $L$-factor, then $H$ has at least $L^n \frac{n!}{n^n}$ perfect matchings.

We now show how Lemmas \ref{lem:typicality} and \ref{lem:Matthew's_method} imply that w.h.p.\ the algorithm does not abort.

	Let $A_i$ be the event that $G_i$ is not $\delta^{1/3}$-pseudorandom, and let $B_i$ be the event that $H_i$ has less than $L(i)^n \frac{n!}{n^n}$ perfect matchings. For convenience, for $1 \leq i \leq m+1$ we define $C_i = \cup_{j<i} B_j$. We want to show that $\Pr \left[ C_{m+1} \right]$, which is the probability that the algorithm aborts, is $\oone$.

	We prove this by induction. We assume that $\Pr [C_i] = \oone$, and prove that $\Pr [C_{i+1}] = \oone$.

	\begin{align*}
		\Pr \left[ C_{i+1} \right] = \Pr [\cup_{j:j\leq i} B_j] \leq \sum_{j \leq i} \Pr [B_j \given \cap_{\ell < j} \overline{B_\ell}] = \sum_{j \leq i} \Pr [B_j \given \overline{C_j}].
	\end{align*}
	Now,
	\[
	\Pr \left[ B_j| \overline{C_j} \right] \leq \Pr \left[B_j \given \overline{C_j},\overline{A_j} \right] + \Pr \left[A_j \given \overline{C_j} \right],
	\]
	and
	\begin{align*}
 		& \Pr \left[ B_j \given \overline{C_j} , \overline{A_j} \right]
        \leq  \frac{1}{ \Pr \left[ \overline{C_j}  \given \overline{A_j} \right]} \Pr \left[B_j \given \overline{A_j} \right].
\end{align*}
Applying Bayes' theorem, this is equal to
\begin{align*}
& \frac{\Pr \left[ \overline{A_j}  \right]}{\Pr \left[\overline{C_j} \right] \Pr \left[\overline{A_j} \given  \overline{C_j}   \right] } \Pr \left[B_j \given \overline{A_j} \right] \leq (1+o(1)) \Pr \left[B_j \given \overline{A_j} \right].
\end{align*}
	The inequality follows from the induction hypothesis and Lemma \ref{lem:Matthew's_method}. Thus,
	\[
	\Pr \left[ C_{i+1} \right] \leq (1+o(1))   \sum_{j \leq i } \Pr \left[ B_j \given \overline{A_j} \right] + \sum_{j \leq i } \Pr \left[A_j \given \overline{C_j} \right].
	\]
	Now, by Lemma \ref{lem:typicality}, $\Pr \left[B_j|\overline{A_j} \right] = n^{-\omegaone}$, and by Lemma \ref{lem:Matthew's_method} we have $\Pr \left[A_j \given \overline{C_j} \right] = e^{-\Omega(n)}$. This implies that $\Pr [C_{i+1}] = \oone$, completing the inductive proof. Thus, w.h.p.\ the algorithm does not abort.

We turn to prove the lemmas. In what follows, we will make repeated use of the following version of Chernoff's inequality.

\begin{thm}[Chernoff's inequality]\label{thm:chernoff}
Let $X_1, ... ,X_N$ be i.i.d.\ Bernoulli random variables with $\Pr(X_i=1)=p$ for all $i$, and let $Z = \sum_{i=1}^N{X_i}$. Then for all $\alpha>0$ it holds that
\begin{align*}
& \Pr[Z<(1-\alpha) N p] \leq \exp\left(-\frac{\alpha^2 N p}{2}\right),\\
& \Pr[Z > (1 + \alpha) N p] \leq \exp \left(-\frac{\alpha^2 N p}{3}\right).
\end{align*}

\end{thm}

\begin{proof}[Proof of Lemma \ref{lem:typicality}:]

We use the following generalization of Hall's theorem, which can be found, e.g., in \cite[Theorem 3]{alon1998perfect}.

\begin{thm}
\label{thm:Hall}
Let $G = \langle U \cup V, E\rangle$ be a balanced bipartite graph on $2n$ vertices. Then $G$ has an $L$-factor if and only if for all $X \subseteq U,Y\subseteq V$ it holds that
\[
E_G(U \setminus X,V \setminus Y) \geq (n-|X|-|Y|)  L.
\]
\end{thm}

Let $X \subseteq U,Y\subseteq V$, and assume without loss of generality that $|X| \geq |Y|$ and that $|X|+|Y|<n$.
We say that the pair $X,Y$ is a \termdefine{bad pair} if the random variable $Z=Z_{X,Y}:=E_H(U \setminus X,V \setminus Y)$ is smaller than $(n-|X|-|Y|)  L$. Our goal is to show that the probability that there exists a bad pair in $H$ is $n^{-\omegaone}$.

We consider three cases.

\begin{itemize}
\item[] \textbf{Case 1:} $n-|X|\geq \frac{\varepsilon}{10} n$ and $|Y|\geq \frac{\varepsilon}{10} n$. Now, $n-|Y|\geq \frac{\varepsilon}{10} n$, because $|Y| \leq |X|$, and so by the pseudorandomness of $G$, we have 
\[
E_G(U \setminus X,V \setminus Y) \geq (1-\delta^{1/3})(n-|X|)(n-|Y|)\frac{k}{n}.
\] 
Hence,
\[
\E [Z] \geq (1-\delta^{1/3})(n-|X|)(n-|Y|)\frac{k p }{n} = \Omega \left( n \log n \right).
\]
Now, we want to bound the probability that $Z <  (1-\delta) k p (n-|X|-|Y|)$.
Note that
\begin{align*}
\begin{split}
(1-\delta) & k p (n-|X|-|Y|) \leq 
\frac{ (1-\delta) k p (n-|X|-|Y|) }{(1-\delta^{1/3})(n-|X|)(n-|Y|)\frac{k p }{n}} \mathbb{E}[Z] \\
& = \frac{1-\delta}{1-\delta^{1/3}}   \frac{n(n-|X|-|Y|)}{(n-|X|)(n-|Y|)}   \mathbb{E}[Z] \\
& = \frac{1-\delta}{1-\delta^{1/3}}   \left(1-\frac{|X||Y|}{(n-|X|)(n-|Y|)}\right)   \mathbb{E}[Z] \\
& \leq \frac{1-\delta}{1-\delta^{1/3}}   \left(1-\left(\frac{\varepsilon}{10}\right)^2\right)   \mathbb{E}[Z]
\leq \left(1 - \frac{\varepsilon^2}{200} \right)   \mathbb{E}[Z].
\end{split}
\end{align*}

The last inequality holds for large enough $n$. Therefore,
\[
\Pr \left[ Z <  (1-\delta) k p (n-|X|-|Y|) \right] \leq \Pr \left[ Z< \left( 1 - \frac{\varepsilon^2}{200} \right)   \mathbb{E}[Z] \right].
\]

By Chernoff's inequality, this is at most
\[
\exp\left( - \frac{1}{2} \left( \frac{\varepsilon^2}{200} \right)^2 \E [Z] \right) =
\exp\left( - \Omega \left( n \log n \right) \right).
\]

As there are less than $4^n$ possible pairs $X,Y$, we can apply a union bound. We conclude that the probability that a pair $X,Y$ considered in this case is bad is at most $\exp \left( - \Omega (n \log n) \right)$.

\item[] \textbf{Case 2:} $Y$ is the empty set. In this case, $Z_{X,Y}$ is the total number of edges with an endpoint in $U \setminus X$, and so it suffices to show that with sufficiently high probability, all degrees are at least $L$.
Indeed, the expected degree of any fixed vertex $v$ is $kp$, so by Chernoff's inequality,
\[
 \Pr \left[ \deg1_H(v) < (1- \delta) kp \right] \leq \exp\left( - \delta^2 k p / 2 \right) \leq n^{ - \omegaone }. 
\]
Therefore, a union bound on all $2n$ vertices implies that the probability that there is a vertex whose degree is less than $L$ is $n^{-\omegaone}$.

\item[] \textbf{Case 3:} Assume now that $t \coloneqq |Y| < \frac{\varepsilon}{10} n$ and $t>0$. Since $|X|+|Y| < n$, we have $t<s:= n - |X|$.

If $s \geq \varepsilon n$, then as the number of edges from $U \setminus X$ to $Y$ is at most $kt$, we have 
\[
E_G(U \setminus X,V \setminus Y) \geq sk-kt = k(s-t).
\]
Therefore, $\mathbb{E}[Z] \geq k(s-t)p \geq s \frac{k p}{2}$.
We want 
\[
Z\geq (1-\delta) kp (n-|X|-|Y|) =  (1-\delta) kp (s-t) \leq (1-\delta)\mathbb{E}[Z].
\]

On the other hand, if $s < \varepsilon n$ then the fact that the number of edges from $U \setminus X$ to $Y$ is at most $st$ implies that

\[
E_G(U \setminus X,V \setminus Y) \geq sk-st = s(k-t).
\]
Therefore, $\mathbb{E}[Z] \geq s(k-t)p \geq s \frac{k p}{2}$.
We want 
\[
Z\geq (1-\delta) kp (n-|X|-|Y|) =  (1-\delta) kp (s-t) \leq
\]
\[
\left(\frac{(1-\delta) kp (s-t)}{s(k-t)p}\right) \mathbb{E}[Z] =
\]
\[
(1-\delta)\frac{1-(t/s)}{1-(t/k)}\mathbb{E}[Z] \leq (1-\delta)\mathbb{E}[Z].
\]

In either case, by Chernoff's inequality, we have
\[
\Pr \left[ Z < (1- \delta) kp (n-|X|-|Y|)\right] \leq
\exp \left( -\delta^2 s \frac{k p}{4} \right) \leq \left( n^{-\omegaone} \right)^s.
\]

We now apply a union bound over all such pairs $X,Y$. Note that $s > t \geq 1$, so the probability that one of the pairs considered in the last two cases is bad is at most
\[
2 \sum_{s=2}^n\binom{n}{s} \sum_{t=1}^{\min(s-1, (\varepsilon/10)n)} \binom{n}{t} \left( n^{- \omegaone } \right)^s \leq 
n^{-\omegaone}.
\]
\end{itemize}
\end{proof}

To prove Lemma \ref{lem:Matthew's_method} we will need the following upper bound on the number of regular bipartite graphs that are not pseudorandom.

\begin{lem}\label{lem:bad graphs bound}
	Let $\varepsilon n \leq k \leq n$. The number of $k$-regular bipartite graphs on $2n$ vertices that are not $\delta^{1/3}$-pseudorandom is bounded from above by:
	\[
	\binom{n}{k}^n \exp \left( - \Omega \left( \delta^{2/3} n^2 \right) \right).
	\]
\end{lem}

\begin{proof}
	Let $R \sim G(n,n;k/n)$ be a balanced, bipartite, binomial random graph with $2n$ vertices, vertex partition $U \cup V$, and edge probability $k/n$. Let $\mB$ be the event that for some $X \subseteq U$, $Y \subseteq V$ satisfying $|X|,|Y| \geq \varepsilon n / 10$, the number of edges between $X$ and $Y$ satisfies $E_R(X,Y) \leq \left( 1 - \delta^{1/3} \right) |X||Y|k/n$. As $E_R(X,Y)$ is distributed binomially with parameters $|X||Y|, k/n$, by Chernoff's inequality and a union bound over all such pairs $X,Y$:
	\[
	\Prob \left[ \mB \right] \leq 4^n \exp \left(- \frac{\delta^{2/3} \varepsilon^3}{200}n^2 \right) =
    \exp \left( - \Omega \left( \delta^{2/3} n^2 \right) \right).
	\]
	On the other hand, let $\mC$ be the event that $R$ is $k$-regular.  By various estimates (e.g.\ \cite[Proposition 2.2]{ordentlich2000two}), the number of $k$-regular bipartite graphs on $2n$ vertices is at least $\binom{n}{k}^{2n} \left( \frac{k}{n} \right)^{kn} \left( 1 - \frac{k}{n} \right)^{n(n-k)}$. As $k$-regular graphs have precisely $kn$ edges:
	\begin{align*}
	\Prob \left[ \mC \right] & \geq \binom{n}{k}^{2n} \left( \frac{k}{n} \right)^{2kn} \left( 1 - \frac{k}{n} \right)^{2n(n-k)} \\
    & = \left( \frac{n!}{k!(n-k)!} \left( \frac{k}{n} \right)^{k} \left( \frac{n-k}{n} \right)^{n-k} \right)^{2n} = \exp \left( - O \left( n \log n \right) \right),
	\end{align*}
    where the final equality follows from Stirling's approximation. Recall that $\delta \geq 1/n$, and so $\delta^{2/3} n^2 = \omega \left( n \log n \right)$. Therefore,
   \[
	\Prob \left[ \mB \given \mC \right] \leq \frac{\Prob \left[ \mB \right]}{\Prob \left[ \mC \right]} \leq \exp \left( - \Omega \left( \delta^{2/3} n^2 \right) \right).
	\]
    Observe that conditioning on $\mC$ gives the uniform distribution on $k$-regular bipartite graphs with $2n$ vertices. As the number of $k$-regular bipartite graphs with $2n$ vertices is bounded from above by $\binom{n}{k}^n$, the lemma follows.
\end{proof}

\begin{proof}[Proof of Lemma \ref{lem:Matthew's_method}]
	The proof is by induction on $i$. Recall that for $1 \leq i \leq m$, $A_i$ denotes the event that $G_i$ is not $\delta^{1/3}$-pseudorandom, $B_i$ is the event that $H_i$ contains fewer than $L(i)^n \frac{n!}{n^n}$ perfect matchings, and $C_i = \cup_{j < i} B_j$. We want to show that for all $1< i \leq m+1$, $\Pr [A_i | \overline{C_i}] = \exp \left( - \Omega (n) \right)$.

	Recall that $k(i)=n-i+1$, and thus $G_i$ is a $k(i)$-regular graph. For $1 \leq j < (1-\varepsilon)n$ define the error function $\alpha(j) = \sum_{s=1}^{j-1} \frac{\log(s)+2}{4s}$. Note that for all $j$, $\alpha(j) \leq \log^2(n)$.
	We will show that the following conditions hold for every $1 \leq j < (1- \varepsilon)n$:
	\begin{enumerate}
		\item For every graph $G$ it holds that
		\[
		\Pr[G_{j} = G \given \overline{C_j}] \leq \binom{n}{k(j)}^{-n} e^{2n ( \delta j + \alpha(j))}.
		\]

		\item $\Pr[ A_j \given \overline{C_j}] = \exp \left( - \Omega \left( n \right) \right)$.

		\item $\Pr[C_{j+1} \given \overline{C_j}] = n^{-\omegaone}$.
	\end{enumerate}
	Observe that $\Prob \left[ \overline{C_1} \right] = 1$. When $j=1$, the first condition holds trivially. The second condition follows from the fact that $G_1 = K_{n,n}$. The third condition follows from Lemma \ref{lem:typicality}. Assume inductively that for $2 \leq i < (1 - \varepsilon )n$, the conditions hold for $j = i - 1$. We will show that they hold for $j=i$. This suffices to prove the lemma.
	
	Let $k=k(i)$. For a bipartite graph $G$ let $\overline{G}$ denote its complement, i.e., the bipartite graph on the same vertex set with all bipartite edges not in $G$. Let $M( \overline{G})$ be the set of perfect matchings in $\overline{G}$.

	\begin{enumerate}
		\item Let $G$ be a $k$-regular bipartite graph. Let $\mu$ be the perfect matching chosen by the algorithm at step $(i-1)$. We apply the law of total probability to the choice of $\mu$. We have $G_i = G$ only if $\mu = \nu$ for some $\nu \in M(\overline{G})$ and $G_{i-1}$ is equal to the union of $G_i$ and $\nu$. Thus:
		\begin{align} \label{eqn:matching_union}
		\Pr[G_i=G \given \overline{C_i}] = \sum_{\nu \in M(\overline{G})} {\Pr[\mu = \nu| G_{i-1}=G \cup \nu, \overline{C_i} ] \Pr [G_{i-1}=G \cup \nu| \overline{C_i} ]}.
		\end{align}
		We bound the probabilities in the sum separately. Once again applying the law of total probability, while observing that conditioning on $\nu \notin M(H_{i-1})$ implies that $\mu \neq \nu$:
		\begin{align*}
			\Pr[ \mu & = \nu \given G_{i-1} = G \cup \nu, \overline{C_i} ] = \\
			& \Prob \left[ \mu = \nu \given \nu \in M(H_{i-1}), G_{i-1} = G \cup \nu, \overline{C_i} \right] \Prob \left[ \nu \in M(H_{i-1}) \given G_{i-1} = G \cup \nu, \overline{C_i} \right].
		\end{align*}
		The event $\overline{C_i}$ implies that $H_{i-1}$ contains at least $L(i-1)^n \frac{n!}{n^n}$ perfect matchings. As $\mu$ is chosen uniformly at random from $M(H_{i-1})$:
		\[
		\Prob \left[ \mu = \nu \given \nu \in M(H_{i-1}), G_{i-1} = G \cup \nu, \overline{C_i} \right] \leq \frac{n^n}{L(i-1)^n n!}.
		\]
		In order to bound the probability that $\nu \in H(M_{i-1})$, we note that every perfect matching contains $n$ edges, and $H_{i-1}$ is a random subgraph of $G_{i-1}$ in which each edge survives with probability $p$. This would suggest a probability of $p^n$. However, we must be careful not to condition on properties of $H_{i-1}$ itself. With this in mind, we replace the conditioning on $\overline{C_i}$ with conditioning on $\overline{C_{i-1}}$, and use the induction hypothesis to obtain:
		\begin{align*}
		\Prob \left[ \nu \in M(H_{i-1}) \given G_{i-1} = G \cup \nu, \overline{C_i} \right]
        &\leq \frac { \Prob \left[ \nu \in M(H_{i-1}) \given G_{i-1} = G \cup \nu, \overline{C_{i-1}} \right]}{\Prob \left[ \overline{C_i} \given \overline{C_{i-1}} \right]} \\
        &\leq \left( 1 + \oone \right) p^n.
		\end{align*}
		Therefore:
		\begin{align}\label{eq:nu}
			\Pr[ \mu & = \nu \given G_{i-1} = G \cup \nu, \overline{C_i} ] \leq \left( 1 + \oone \right) \frac{n^n}{L(i-1)^n n!} p^n.
		\end{align}
		Using the induction hypothesis, we bound the second probability in inequality \eqref{eqn:matching_union} as follows:
		\begin{align}\label{eq:g_i-1}
			\Prob \left[ G_{i-1} = G \cup \nu \given \overline{C_i} \right] \leq
			\frac{\Prob \left[ G_{i-1} = G \cup \nu \given \overline{C_{i-1}} \right]}{\Prob \left[ \overline{C_i} \given \overline{C_{i-1}} \right]} \leq
			\left( 1 + \oone \right) \binom{n}{k + 1}^{-n} e^{2n \left( \delta (i - 1) + \alpha (i-1) \right)}.
		\end{align}
		Together, \eqref{eqn:matching_union}, \eqref{eq:nu}, and \eqref{eq:g_i-1} imply:
		\[
		\Pr[G_i=G \given \overline{C_i}] \leq \left( 1 + \oone \right) \left| M \left( \overline{G} \right) \right| \frac{n^n p^n}{L(i-1)^n n!} \binom{n}{k+1}^{-n} e^{2n \left( \delta (i - 1) + \alpha (i-1) \right)}.
		\]
		Finally, we bound $\left| M \left( \overline{G} \right) \right|$ by using Br\'egman's permanent inequality \cite{bregman1973some}. It implies that the number of perfect matchings in a $d$-regular bipartite graph on $2n$ vertices is at most $\left( d! \right)^{n/d}$. Since $\overline{G}$ is an $(i-1)$-regular bipartite graph on $2n$ vertices, we have $\left| M \left( \overline{G} \right) \right| \leq \left( \left( i-1 \right)! \right)^{n/(i-1)}$. Therefore, using the inequality $\sqrt{2 \pi \ell} (\ell/e)^\ell \leq \ell! \leq e\sqrt{\ell} (\ell/e)^\ell$, which holds for all natural $\ell$:
		\begin{align*}
			\Pr[G_i=G \given \overline{C_i}] & \leq \left( 1 + \oone \right) \left( \left( i-1 \right)! \right)^{n/(i-1)} \frac{n^n p^n}{L(i-1)^n n!} \binom{n}{k+1}^{-n} e^{2n \left( \delta (i - 1) + \alpha (i-1) \right)} \\
			& \leq \left( 1 + \oone \right) \left( e \sqrt{i-1} \right)^{n/(i-1)} \left( \frac{n}{e} \right)^n \frac{1}{n! (1-\delta)^n} \binom{n}{k}^{-n} e^{2n \left( \delta (i - 1) + \alpha (i-1) \right)} \\
            & \leq \binom{n}{k}^{-n} e^{2n \left( \delta i + \alpha (i) \right)}
		\end{align*}
		as desired.

		\item Let $\mG$ be the set of $k$-regular graphs on $2n$ vertices that are not $\delta^{1/3}$-pseudorandom. Then:
		\[
		\Prob \left[ A_i \given \overline{C_i} \right] = \sum_{G \in \mG} \Prob \left[ G_i = G \given \overline{C_i} \right].
		\]
		We have already shown that for any $G$ it holds that $\Prob \left[ G_i = G \given \overline{C_i} \right] \leq \binom{n}{k}^{-n} e^{2 n (\delta i + \alpha(i))}$. Furthermore, by Lemma \ref{lem:bad graphs bound}: $\left| \mG \right| \leq \binom{n}{k}^n \exp \left( - \Omega \left( \delta^{2/3} n^2 \right) \right)$.

		Therefore:
		\[
		\Prob \left[ A_i \given \overline{C_i} \right] \leq \exp \left( 2 \delta n^2 - \Omega \left( \delta^{2/3} n^2 \right) + n \log^2(n) \right) = \exp \left( - \Omega (n) \right).
		\]

		\item We have:
		\[
		\Prob \left[ C_{i+1} \given \overline{C_i} \right]
		\leq \Prob \left[ C_{i+1} \given \overline{C_i}, \overline{A_i} \right] \Prob\left[ \overline{A_i} \given \overline{C_i} \right] + \Prob \left[ C_{i+1} \given \overline{C_i}, A_i \right] \Prob \left[ A_i \given \overline{C_i} \right].
		\]
		We have already shown that $\Prob \left[ A_i \given \overline{C_i} \right] = n^{-\omegaone}$. Furthermore, Lemma \ref{lem:typicality} implies that $\Prob \left[ C_{i+1} \given \overline{C_i}, \overline{A_i} \right] =n^{-\omegaone}$. Therefore:
		\[
		\Prob \left[ C_{i+1} \given \overline{C_i} \right] = n^{-\omegaone}.
		\]
	
	\end{enumerate}

\end{proof}

\section{Proof of Theorem \ref{thm:list_coloring}}\label{sec:boxes}

\newcommand{\rc}{{\left(r,c\right)}}
\newcommand{\Mnm}{{M_{n,n,m}}}

Let $M_{n,m,k}$ be the set of all $n \times m \times k$ 0-1 arrays. For ${M \in M_{n,m,k}}$ we denote by $\left|M\right|$ the number of $1$s in $M$.

\begin{definition}
	For integers $n,m \in \N$, an \termdefine{$\left( n,n,m \right)$-array process} is a sequence $\left\{ M_i \right\}_{i=0}^{n^2m} \subseteq \Mnm,$ where $M_0$ is the all $0$s array, and $M_{i+1}$ is obtained from $M_i$ by changing a single $0$ to $1$.
\end{definition}

We denote a generic $\left( n,n,m \right)$-array process by $ \tilde{M} = \left\{ M_i \right\}_{i=0}^{n^2 m} $ and write $\AP{n}{n}{m}$ for the uniform distribution on such processes.

\begin{definition}
	Let $Q$ be a non-trivial monotone increasing property of $\Mnm$, and let $\tilde{M}$ be an array process. The \termdefine{hitting time} of $Q$ w.r.t.\ $\tilde{M}$ is defined as:
	\[
	\tau \left(\tilde{M}; Q \right) = \min \left\{ t : M_t \text{ has } Q \right\}.
	\]
\end{definition}

We are interested in the hitting time for the property that $\tilde{M} \sim \AP{n}{n}{m}$, where $m \geq n$, supports a Latin box.

Let $M \in \Mnm$. For $1 \leq r,c \leq n$ we refer to a line of the form $\left(r,c, \cdot \right)$ as a \termdefine{shaft}. The shaft is \termdefine{empty} if $M \left( r,c,1 \right) = M \left( r,c,2 \right) = \ldots = M \left( r,c,m \right) = 0$. Clearly, a necessary condition for $M$ to support a Latin box is that it have no empty shafts. We show that for $m$ slightly larger than $n$ this is asymptotically almost surely a sufficient condition.

\begin{thm}\label{thm:hitting_time}
	For every $\varepsilon>0$, if $\tilde{M} \sim \AP{n}{n}{\left( 1 + \varepsilon \right) n}$ then asymptotically almost surely:
	\[
	\tau \left( \tilde{M} \text{ has no empty shafts} \right) = \tau \left( \tilde{M} \text{ supports a Latin box} \right).
	\]
\end{thm}

Theorem \ref{thm:list_coloring} follows from a standard coupling between random processes and their binomial counterparts.

\begin{proof}[Proof of Theorem \ref{thm:list_coloring}]
	\newcommand{\thresh}{{\frac{2}{1 + \varepsilon} \frac{\log n}{n}}}
    
	Let $\varepsilon > 0$ and let $M \sim \mnp{n}{n}{(1+\varepsilon)n}{p}$. For a fixed pair $r$ and $c$, the probability that $M(r,c,\cdot)$ is empty is $(1-p)^{(1+\varepsilon)n}$. The different shafts are independent, and so the probability that there are no empty shafts is 
    \[
    q(p) = (1-(1-p)^{(1+\varepsilon)n})^{n^2}.
    \]
 If, for some $\delta>0$, $p \leq (1-\delta) \thresh$, then $q(p) \rightarrow 0$, and therefore w.h.p.\ $M$ contains empty shafts. In this case $M$ does not support a Latin box.

On the other hand, if $p \geq (1+\delta) \thresh$, then $q(p)\rightarrow 1$, and so w.h.p.\ $M$ contains no empty shafts. Consider the following random process. For each triple $r,c,v$ of indices, choose a real number $\alpha_{r,c,v} \sim U[0,1]$ uniformly at random from the interval $[0,1]$, all choices independent. Now, $M$ is identically distributed to the array $M'$ in which all entries with $\alpha_{r,c,v}<p$ are set to $1$, and all other entries are $0$. Furthermore, let $\tilde{N'}$ be the array process obtained by flipping the entries of the all zeros array to $1$ in ascending order of $\alpha$. Note that $\tilde{N'}$ is a uniformly random array process.

Let $t = |M'|$. Observe that $M' = N'_t$, and therefore, w.h.p.\ $N'_t$ contains no empty shafts. Thus, by Theorem \ref{thm:hitting_time}, w.h.p.\ $N'_t$ contains a Latin box, which implies that $M'$ contains a Latin box. Since $M$ and $M'$ are identically distributed, w.h.p.\ $M$ contains a Latin box. 
\end{proof}

Henceforth, fix $\varepsilon > 0$ and $m = \left(1 + \varepsilon\right)n$. Wherever necessary we assume that $n$ is arbitrarily large and $\varepsilon$ is arbitrarily small.

To prove Theorem \ref{thm:hitting_time}, we introduce a new model for random arrays, denoted $\geqnp$, whose sample space consists of $n \times n \times m$ $\left(0,1\right)$-arrays where each $1$ is colored either green or blue. The green values are an array $M_G \sim \mnp{n}{n}{m}{p}$. Then, from each empty shaft in $M_G$, a position is chosen uniformly at random (all choices independent), changed to $1$, and colored blue. Denote by $M_B$ the array of blue values, and set $M = M_G + M_B$. The next proposition shows that it is enough to prove that w.h.p. $M \sim \geqnp$ supports a Latin box, for a suitable choice of $p$.

\begin{pr}\label{pr:green_blue}
	Let $Q$ be a monotone property of $\Mnm$ implying that there are no empty shafts. Let $p = \frac{2}{1 + \varepsilon} \frac{\log n - \log \log n}{n}$. If $Q$ holds w.h.p.\ for $M \sim \geqnp$, then for almost every $\tilde{M} \sim \AP{n}{n}{m}$:
	\[
	\tau \left(\tilde{M} ; Q \right) = \tau \left( \tilde{M} \text{ has no empty shafts} \right).
	\]
\end{pr}

Proposition \ref{pr:green_blue} is similar to analogous claims used to prove hitting time results in random graph and hypergraph processes (for example \cite[Lemma 7.9]{bollobas1998random} and \cite[ Lemma 1]{krivelevich1996perfect}).

\begin{proof}
	As in the proof of Theorem \ref{thm:list_coloring}, for each triple $r,c,v$ let $\alpha_{r,c,v} \in [0,1]$ be drawn uniformly at random and independently. Now $M_G$ is identically distributed to the array $M'$ in which all entries with $\alpha_{r,c,v} < p$ are set to $1$, and all other entries are $0$. Furthermore, $M_B$ is identically distributed to the array $M''$ in which, for each empty shaft $r,c$ in $M'$, the element with minimal $\alpha_{r,c,v}$ is set to $1$. As before, let $\tilde{N}$ be the (uniformly random) array process where elements are set to $1$ in ascending order of $\alpha$. Let $N_t$ be the first array in which there are no empty shafts. Recall that w.h.p.\ $M'$ has empty shafts. Therefore, w.h.p., $supp (M' + M'') \subseteq supp (N_t)$. Now, $M'+M'' \sim M$, w.h.p.\ $M \in Q$, and $Q$ is a monotone property. Therefore, $N_t \in Q$ w.h.p.
\end{proof}

Henceforth, let $M = M_G + M_B \sim \geqnp$, with $p$ as in the statement of Proposition \ref{pr:green_blue}. Unless stated otherwise all probabilities refer to this distribution. For $1\leq r,c \leq n$ set:
\begin{align*}
\begin{split}
& d \rc = \sum_{i=1}^{m} M\left(r,c,i\right), \\
& d_m \rc = \sum_{i=n+1}^{m}M\left(r,c,i\right).
\end{split}
\end{align*}

In what follows, we think of an $n \times n \times m$ Latin box as a function $L:[n]^2 \rightarrow [m]$ such that $L \left(a,b\right) \neq L \left(c,d\right)$ whenever $\left(a,b\right)$ and $\left(c,d\right)$ have exactly one coordinate in common. A function $B : S \rightarrow \left[m\right]$ is a \termdefine{partial Latin box} if $S \subseteq \left[n\right]^2$ and $B \left(a,b\right) \neq B \left(c,d\right)$ whenever $\left(a,b\right)$ and $\left(c,d\right)$ have exactly one coordinate in common. We call the positions in $S$ \termdefine{covered}, and those in $\left[n\right]^2 \setminus S$ \termdefine{uncovered}. $B$ is \termdefine{supported} by $M$ if for all $\rc \in S, M \left( r,c,B\rc \right) = 1$. We will occasionally use the adjective ``proper'' to distinguish a Latin box from a partial one.

Assuming Proposition \ref{pr:green_blue}, it suffices to prove that w.h.p.\ $M$ supports a Latin box. We will show that w.h.p.\ we can construct partial Latin boxes $B_1, B_2, B_3, B_4$ supported by $M$, and then show that w.h.p.\ $B_4$ can be completed to a proper Latin box $B$, also supported by $M$. The stages of the construction are roughly as follows:

\begin{itemize}
	
	\item Construct $B_1$ by covering all positions $\rc$ s.t.\ ${d \rc - d_m \rc \leq \log\log n}$ and $d_m \rc \leq \frac{\varepsilon}{1+\varepsilon}\log n$.
	\item Extend $B_1$ to $B_2$ by covering all positions $\left(r,c\right)$ for which $d_m \rc \leq \frac{\varepsilon}{1+\varepsilon}\log n$.
	\item Construct $B_3$ using only symbols from $\left[n\right]$, and covering all but $o \left(n\right)$ positions in each row and column.
	\item Combine $B_2$ and $B_3$ to construct $B_4$, in which all but $o \left(n\right)$ positions in each row and column are covered, and in addition each uncovered position $\left(r,c\right)$ satisfies ${d_m \rc \geq \frac{\varepsilon}{1+\varepsilon} \log n}$.
	\item Extend $B_4$ to a proper Latin box $B$ by covering the remaining positions with values from $\left\{n+1, \ldots, m \right\}$.
	
\end{itemize}

$B_1, B_2$, and $B$ are found via a simple randomized algorithm. To construct $B_3$, we use a random greedy algorithm. $B_4$ is constructed by ``overwriting'' $B_3$ with $B_2$, and erasing any values from $B_3$ that collide with $B_2$. We now prove that these steps can, in fact, be successfully completed w.h.p.

The following lemma constructs $B_2$. The construction of $B_1$ is an ingredient in the proof.

\begin{lemma}\label{lem:low_deg}
	W.h.p.\ $M$ supports a partial Latin box $B_2$ covering only $o \left(n\right)$ positions in each row and column s.t. if $\rc \in \left[n\right]^2$ is not covered by $B_2$ then ${d_m \rc \geq \frac{\varepsilon}{1 + \varepsilon} \log n}$.
\end{lemma}

\begin{proof}
	
	For a position $\rc$ let $X_{r,c} = \sum_{i=n+1}^mM_G\left(r,c,i\right)$, and let $Y_{r,c}$ be the indicator of the event $X_{r,c} < \frac{\varepsilon}{1 + \varepsilon} \log n$. Then $X_{r,c}$ are i.i.d.\ binomial random variables with distribution $Bin \left(\varepsilon n, p \right)$, and so by Chernoff's inequality (Theorem \ref{thm:chernoff}):
	\[
	\Prob \left[X_{r,c} \leq \frac{\varepsilon}{1+\varepsilon} \log n \right] \leq n^{ - \frac{\varepsilon}{4 \left( 1 + \varepsilon \right)} + \oone}.
	\]
	Thus $Y_{r,c} \sim Ber \left(q\right)$ for some $q \leq n^{ - \frac{\varepsilon}{4 \left( 1 + \varepsilon \right)} + \oone}$. Now, the expected number of positions in each row or column for which $Y_{r,c} = 1$ is $nq \leq n^{1 - \frac{\varepsilon}{4 \left(1 + \varepsilon \right)} + \oone}$, and again applying Chernoff's inequality we obtain that w.h.p.\ there are at most $n^{1-\delta}$ such positions in each row and column, for some $\delta > 0$.
	
	Let $S = \left\{\rc \in \left[n\right]^2 : Y_{r,c} = 1 \right\}$. By the above, w.h.p.\ $S$ contains only $o \left(n\right)$ positions in each row and column. We show that w.h.p.\ we can find a partial Latin box $B_2$ supported by $M$ whose domain is $S$.
	
	We do this in two stages: We first cover all $\rc \in S$ s.t.\ $d\rc - d_m \rc$ is small with a partial Latin box $B_1$. We then show that w.h.p.\ $B_1$ can be extended to the desired $B_2$.
	
	Let $T = \left\{ \rc \in S : d\rc - d_m \rc \leq \log \log n \right\}$. Observe that
    \[
    d \rc - d_m \rc = \sum_{i=1}^n M \left( r,c,i \right) \geq \sum_{i=1}^n M_G \left( r,c,i \right) \sim Bin \left( n,p \right).
    \]
    We have:
	\begin{equation*}
	\begin{split}
	\Prob \left[ \rc \in T \given \rc \in S \right]
	& \leq \sum_{k=0}^{\log \log n} \binom{n}{k} p^k \left(1 - p\right)^{n - k} 
	= \left(1 - p\right)^{n} \sum_{k=0}^{\log \log n} \binom{n}{k} \left(\frac{p}{1-p}\right)^k \\
	& \leq \left( \frac{\log n}{n} \right)^{\frac{2}{1 + \varepsilon}} \left( 1 + \sum_{k=1}^{\log\log n} \left(\frac{2 e \log n}{k}\right)^k \right) \\
	& \leq \left( \frac{\log n}{n} \right)^{\frac{2}{1 + \varepsilon}} \log\log n \left(6 \log n \right)^{\log\log n}
    = \frac{ e^{O \left( \left( \log\log n \right)^2 \right) } }{ n^{\frac{2}{ 1 + \varepsilon } }}.
	\end{split}
	\end{equation*}
	
	Applying Markov's inequality we conclude that w.h.p.\ $\left|T\right| \leq n^{3 \varepsilon}$.
	
	We construct $B_1$ by covering $T$. Note that for every $\rc \in T$, $d\rc \geq 1$. For each $\rc \in T$, choose $B_1\rc$ uniformly at random from $\left\{ i : M\left(r,c,i\right) = 1 \right\} \cap \left[n\right]$ if this set is non-empty; otherwise choose $B_1\rc$ uniformly at random from $\left\{ i : M\left(r,c,i\right) = 1 \right\} \subseteq \left[m\right] \setminus \left[n\right]$. Note that in the former case $B_1 \rc$ is distributed uniformly amongst $\left[n\right]$ and in the latter case $B_1 \rc$ is distributed uniformly amongst $\left[m\right] \setminus \left[n\right]$. Therefore, $\left\{ B_1\rc \right\}_{\rc \in T}$ is a collection of $O \left(n^{3 \varepsilon}\right) = o \left(\sqrt{n}\right)$ values, each chosen uniformly at random and independently from a set of size $\Omega \left( n \right)$. Hence w.h.p.\ no value appears more than once. This implies that w.h.p.\ $B_1$ is indeed a partial Latin box covering $T$.
	
	The remaining positions $\rc \in S\setminus T$ all satisfy $d \rc - d_m \rc \geq \log\log n$. For each $\rc \in S \setminus T$ let $V' \rc := \left\{ i \in \left[n\right] : M\left(r,c,i\right) = 1 \right\}$. Choose $V \rc \subseteq V'\rc$ of size $\log\log n$ uniformly at random and independently. Note that $\{V\rc\}_{\rc \in S\setminus T}$ is a collection of uniformly random and independent elements of $\binom{[n]}{\log\log n}$. We construct $B_2$ by extending $B_1$ greedily while avoiding collisions: For each $\rc \in S \setminus T$, we choose $B_2 \rc$ uniformly at random from the values in $V \rc$ that have not yet been used in row $r$ or column $c$. W.h.p.\ this procedure succeeds: Indeed, when choosing the value of $B_2$ for any heretofore uncovered $\left(r,c\right) \in S$, there are at most $2 n^{1 - \delta} + n^{3\varepsilon} \leq 3n^{1-\delta}$ previously covered positions in row $r$ and column $c$. Thus there are at most $3 n^{1-\delta}$ forbidden values. Therefore the probability that $V \rc$ contains only forbidden values is at most:
	\[
	\frac{\binom{3 n^{1-\delta}}{\log \log n}}{\binom{n}{\log\log n}}
	= n^{- \omegaone}.
	\]
	Applying a union bound to the $O\left(n^2\right)$ steps in the greedy algorithm, we see that w.h.p.\ the algorithm succeeds in constructing $B_2$.
	
\end{proof}

To construct $B_3$ we need the following lemma.

\begin{lem}\label{lem:packing}
	Let $q = \omega \left( \frac{1}{n} \right)$ and let $M \sim \mnp{n}{n}{n}{q}$. W.h.p.\ $M$ supports a partial Latin box with at most $o \left(n\right)$ uncovered positions in each row and column.
\end{lem}

We prove Lemma \ref{lem:packing} by showing that w.h.p.\ a random greedy algorithm succeeds in finding an appropriate partial Latin box. The proof is deferred to Appendix \ref{app:rga}.

\begin{proof}[Proof of Theorem \ref{thm:hitting_time}]
	
	Recall that $M = M_G + M_B \sim \geqnp$. W.h.p.\ $M$ supports a partial Latin box $B_2$ as per the conclusion of Lemma \ref{lem:low_deg}. By Lemma \ref{lem:packing}, w.h.p.\ $M_G \sim \mnp{n}{n}{m}{p}$ supports a partial Latin box $B_3$ covering all but at most $o \left(n\right)$ positions in each row and column, and using only values from $\left[ n \right]$.
	
	For $i=2,3$ let $S_i$ be the set of positions covered by $B_i$. Define the partial Latin box $B_4$ as follows: For all $\rc \in S_2$ set $B_4 \rc = B_2 \rc$. For all $\rc \in S_3 \setminus S_2$ s.t.\ $B_3 \rc$ isn't used by $B_2$ in row $r$ or column $c$, set $B_4 \rc = B_3 \rc$. $B_4$ is thus a partial Latin box covering all but at most $o \left( n \right)$ positions in each row and column, and in which each row and column uses at most $o \left( n \right)$ values from $\left\{n+1,n+2,\ldots, m \right\}$. Additionally, if $\rc$ isn't covered by $B_4$ then (since $\rc$ is not covered by $B_2$) $d_m \rc \geq \frac{\varepsilon}{1+\varepsilon} \log n$. We now show that w.h.p.\ a random greedy algorithm succeeds in extending $B_4$ to a proper Latin box.
	
	In a manner similar to the proof of Lemma \ref{lem:low_deg}, for each uncovered $\rc$ let $W' \rc = \left\{ v \in \left[m\right] \setminus \left[ n \right] : M \left(r,c,v\right) = 1 \right\}$, and let $W \rc \subseteq W' \rc$ be uniformly random subsets of size $\frac{\varepsilon}{1+\varepsilon} \log n$ chosen independently.
    
    Iterate over the uncovered elements in an arbitrary order. For every uncovered $\rc$ choose $B \rc$ uniformly at random from $W \rc$ that have not previously been used in the same row or column.  At each step of the algorithm, there are at most $o \left(n\right)$ forbidden values, so the probability that all available values are forbidden is at most:
	\[
	\frac{\binom{o \left(n\right)}{\frac{\varepsilon}{1+\varepsilon} \log n }}{\binom{\varepsilon n}{\frac{\varepsilon}{1+\varepsilon} \log n }}
	= n^{-\omegaone}.
	\]
	
	There are $O \left(n^2\right)$ steps in the greedy algorithm so, applying a union bound, the probability of failure is $\oone$.
	
\end{proof}

\subsection*{Acknowledgment} We wish to thank the anonymous referees for their helpful comments, which greatly improved the exposition, as well as for finding an error in a previous version of the paper.

\appendix

\section{Random Greedy Packing in Random Hypergraphs}\label{app:rga}

\newcommand{\hthree}{{H_3 \left( n \right)}}

In this section we prove Lemma \ref{lem:packing}. Although it is a statement about random arrays, it is convenient to reformulate it in terms of random hypergraphs. This is because the random greedy algorithm we are about to introduce is similar to the \termdefine{triangle removal process} analyzed, among others, by Wormald \cite[Section 7.2]{Wo99} and Bohman, Frieze, and Lubetzky \cite{bohman2015random}.

\subsection{Notation and Terminology}

We denote by $\hthree$ the set of tripartite, $3$-uniform hypergraphs whose vertex set is $\left[n\right] \sqcup \left[n\right] \sqcup \left[n\right]$. A \termdefine{triangle} is a partite vertex set of size $3$ and an \termdefine{edge} is a partite vertex set of size $2$. To avoid unnecessary delimiters we sometimes write $abc$ for the triangle $\{a,b,c\}$, and $ab$ for the edge between $a$ and $b$. We denote by $\mathcal{H}_3 \left(n;p\right)$ the distribution on $\hthree$ where each triangle is included in the hypergraph with probability $p$, independently of the other triangles, and we denote by $\mathcal{H}_3 \left( n ; m \right)$ the distribution on $\hthree$ where the triangle set is a uniformly random element of $\binom{\left[n\right]^3}{m}$.

Let $H \in \hthree$, and let $T \left(H\right)$ denote the set of its triangles. A set $S \subseteq T \left(H\right)$ is a \termdefine{set of edge-disjoint triangles} (\termdefine{SET}) in $H$ if for all $t_1,t_2 \in S$, $\left|t_1 \cap t_2\right| \geq 2 \implies t_1 = t_2$. If $uv$ is an edge, we say it is \termdefine{covered} by $S$ if there exists some $t \in S$ s.t.\ $\left\{ u,v \right\} \subseteq t$. In this case we write $uv \in G \left(S\right)$. We say a triangle $t$ is \termdefine{edge-disjoint from $S$} if none of its edges are covered by $S$. For $v \in V \left(H\right)$, let $d_S \left(v\right)$ be the number of triangles in $S$ containing $v$.

For $a,b \in \R$, we write $ a \pm b$ to indicate some quantity in the interval $\left[ a - \left|b\right| , a + \left|b\right| \right]$.
We say that an event occurs \termdefine{with very high probability} (\textbf{w.v.h.p.}) if it occurs with probability $1 - n^{-\omegaone}$.

\subsection{From Arrays to Hypergraphs}

Let $M \in M_{n,n,n}$. We define the hypergraph $H_M \in \hthree$ by setting $T (H_M) = \left\{ (i,j,k) \in [n]^3 : M(i,j,k) = 1 \right\}$. This induces a natural correspondence between SETs in $H_M$ and partial Latin boxes supported by $M$.

Lemma \ref{lem:packing} now follows from: 

\begin{lemma}\label{lem:hg_packing}
	Let $p = \omega \left( \frac{1}{n} \right)$ and let $H \sim \mathcal{H}_3 \left(n;p\right)$. W.h.p.\ $T\left( H \right)$ contains an SET $S$ s.t.\ for every vertex $v$, $d_S \left( v \right) = \left(1 - \oone \right) n$.
\end{lemma}

\subsection{Proof of Lemma \ref{lem:hg_packing}}

In the \termdefine{random hypergraph process}, the triangles of the complete tripartite $3$-uniform hypergraph $K^{\left(3\right)}_{n,n,n}$ are considered one by one in a uniformly random order $t_1, t_2, \ldots, t_{n^3}$. 
This process generates a sequence of hypergraphs $H_0, H_1, \ldots, H_{n^3} \in \hthree$, where $T \left(H_0\right) = \emptyset$ and $T \left( H_{i+1} \right) = T \left( H_i \right) \cup \left\{ t_{i+1} \right\}$. We couple this with the following process: $S_0 = \emptyset$, and $ S_{i+1} = S_i \cup \left\{t_{i+1}\right\}$ if $t_{i+1}$ is edge disjoint from $S_i$, and $S_{i+1} = S_i$ otherwise. Observe that for every $i$, $S_i$ is an SET in $H_i$.
The next proposition says that w.v.h.p.\ the vertex degrees in $S_i$ are concentrated.

\begin{pr}\label{pr:greedy}
	There exists some $\delta > 0$ s.t.\ w.v.h.p.\ for every $v \in \left[n\right] \sqcup \left[n\right] \sqcup \left[n\right]$ and every $0 \leq m \leq n^{2+\delta}$:
	\[
	d_{S_m} \left( v \right) = \left( 1 - \left(1 \pm \oone \right) \frac{1}{\sqrt{1+2\frac{m}{n^2}}}\right)n .
	\]
\end{pr}

Before proving Proposition \ref{pr:greedy}, we first describe how Lemma \ref{lem:hg_packing} follows from Proposition \ref{pr:greedy}.

The \termdefine{random greedy packing algorithm in $H \in \hthree$} is the following probabilistic procedure: Set $S = \emptyset$. As long as there are triangles in $T\left(H\right)$ that are edge disjoint from $S$, choose one uniformly at random and add it to $S$. If there are no such triangles, halt.
Say that a hypergraph in $\hthree$ is \termdefine{good} if it satisfies the conclusion of Lemma \ref{lem:hg_packing}.
Let  $H \sim \mathcal{H}_3 \left(n;p\right)$. We claim that w.h.p.\ $H$ is good, and this is witnessed by the result of the random greedy packing algorithm in $H$.

Clearly, the distribution of $H$ conditioned on $|T(H)| = m$ is identical to $H_m$. Moreover, given $H_m$, $S_m$ is distributed identically to the result of the random greedy packing algorithm in $H_m$. Note also that the probability that $H_m$ is good is increasing in $m$. 
As $|T(H)|\sim Bin \left(n^3,p\right)$ and $p = \omega \left(\frac{1}{n}\right)$, there exists some $k = \omega \left( n^2 \right)$, s.t.\ w.h.p.\ $|T(H)| \geq k$. Proposition \ref{pr:greedy} implies that $H_k$ is good w.v.h.p.
Therefore,
\begin{align*}
\Pr [H \text{ is good}] \geq & \sum_{m = k}^{n^3} {\Pr [H \text{ is good} \given |T(H)| = m] \Pr [ |T(H)| = m] } \\
= & \sum_{m=k}^{n^3} {\Pr [H_m \text{ is good}] \Pr [|T(H)|=m]} \\
\geq & \Pr [ |T(H)| \geq k] \Pr [H_k \text{ is good}]  = 1-o(1).
\end{align*}

We turn to prove Proposition \ref{pr:greedy}.

\begin{proof}
	
	We prove the proposition for $\delta = \frac{1}{100}$.
    
    In the spirit of the differential equation method of Wormald \cite{Wo99} we track a set of random variables throughout the hypergraph process by modeling their evolution on a system of differential equations.
    
    We make use of the following version of the Azuma-Hoeffding inequality, which follows from \cite[Lemma 4.2]{Wo99}.
    \begin{lem}\label{lem:azuma}
    	Let $A_0 \subseteq A_1 \subseteq \ldots \subseteq A_N$ be a filtration of a finite probability space. Let $X_0 , X_1 , \ldots , X_N$ be a sequence of random variables s.t.\ for every $i$, $X_i$ is $A_i$-measurable. Assume that for some $C>0$, $\left| X_{i+1} - X_i \right| \leq C$ for all $i$. Assume further that for all $i$, $\E \left[ X_{i+1} - X_i \given A_i \right] \leq 0$, i.e.\ $X_0, X_1 , \ldots, X_N$ is a supermartingale. Finally, assume $X_0 \leq 0$. Then, for all $\lambda >0$:
        \[
        \Prob \left[ X_N > \lambda \right] \leq \exp \left( - \frac{\lambda^2}{2 N C^2} \right) .
        \]
    \end{lem}
	
	We define the following functions on $\left[0,\infty\right)$, whose relevance will become apparent presently:
	
	\begin{align*}
	\begin{split}
	&y \left(x\right) = \frac{1}{\sqrt{1+2x}} \\
	&z \left(x\right) = \frac{1}{1+2x} 
	\end{split}
	\end{align*}
	
	These satisfy the differential equations:
	
	\begin{align*}
	\begin{split}
	& y' = - yz \\
	& z' = -2 z^2 
	\end{split}
	\end{align*}
	
	We now define the variables we want to track. For every vertex $v$ and $0 \leq i \leq n^{2+\delta}$ write:
	\[
	c_v \left( i \right) = n - d_{S_i} \left( v \right) .
	\]
	
	Next, for $0\leq i \leq n^{2+\delta}$ we define the set of \termdefine{permissible} triangles:
	\[
    A_i = \left\{ t_j : i < j \leq n^3, \forall t \in S_i, \left|t \cap t_j\right| \leq 1 \right\} .
    \]
	In words, $A_i$ is the set of triangles not in $H_i$ that, if selected at time $i+1$, will be included in $S_{i+1}$.
	
	For every uncovered edge $uv$, we track the number of permissible triangles containing it. For convenience, we associate a random variable to covered edges as well:
	\[
	d_{uv} \left( i \right) = 
	\begin{cases}
	\left|\left\{ t \in A_i : \left\{ u,v \right\} \subseteq t \right\} \right| & uv \notin G \left( S_i \right) \\
	nz \left( \frac{i}{n^2} \right) & otherwise
	\end{cases}
    .
	\]
	
	We will show that w.v.h.p.\ for every $0 \leq i \leq n^{2 + \delta}$, every vertex $v$, and every edge $uv$:
	\begin{align}\label{eq:goal}
	\begin{split}
	& c_v \left( i \right) = \left( 1 \pm \oone \right)ny \left(\frac{i}{n^2}\right) \\
	& d_{uv} \left( i \right) = \left( 1 \pm \oone \right)nz \left(\frac{i}{n^2}\right) .
	\end{split}
	\end{align}
	In particular, this will prove the proposition.
	
	We first consider the evolution of the random variables $d_{uv}$. Note that if $uv$ is covered by $S_i$ then (by definition) $d_{uv} \left( i \right) = n z \left(\frac{i}{n^2}\right)$ and there is nothing to prove. So assume that $uv$ is not covered by $S_i$. How might $d_{uv}$ change during step $i+1$? Well, if $uv$ remains uncovered, then $d_{uv}$ will change if and only if some permissible triangle $t \in A_i$ containing $uv$ is no longer in $A_{i+1}$. Now, $uv \subseteq t \in A_i$ will not be in $A_{i+1}$ if $\left| t_{i+1} \cap t \right| = 2$, and $t_{i+1} \in A_i$. In this case, $d_{uv}$ decreases by 1. For every $uvw \in A_i$ there are $d_{uw} \left(i\right) + d_{vw} \left(i\right) - 2$ triangles in $A_i$ that have this effect. Thus, observing that at step $i$ there are $\left( 1 - O \left( n^{\delta - 1} \right) \right) n^3$ triangles remaining to be considered that do not contain $uv$:
	\begin{align}\label{eq:prob}
	\begin{split}
	\Prob & \left[ d_{uv}\left( i+1 \right) \neq d_{uv}\left( i \right) \given H_i,S_i, uv \notin G \left(S_{i+1} \right) \right]
	\\ & = \frac{1}{\left( 1 - O \left( n^{\delta - 1} \right) \right) n^3} \sum_{uvw \in A_i} \left( d_{uw}\left( i \right) + d_{vw}\left( i \right) - 2 \right)
	 \leq \frac{4}{n} .
	\end{split}
	\end{align}
	Note that since the underlying graph is tripartite, so long as $uv$ remains uncovered, $d_{uv}$ can decrease by at most $1$ in a single step. Therefore:
	\begin{align}\label{eq:expec}
    	\begin{split}
	\E & \left[ d_{uv}\left( i+1 \right) - d_{uv}\left( i \right) \given H_i,S_i, uv \notin G \left(S_{i+1} \right) \right]
	\\ & = - \Prob \left[ d_{uv}\left( i+1 \right) \neq d_{uv}\left( i \right) \given H_i,S_i, uv \notin G \left(S_{i+1} \right) \right] .
    \end{split}
	\end{align}
    
    Lemma \ref{lem:azuma} (the Azuma--Hoeffding inequality) requires control over the maximal one-step change of the sequence of random variables. Although the maximum change in $d_{uv}$ is $1$, this is too large for Lemma \ref{lem:azuma} to be useful. Therefore, we show that $d_{uv}$ cannot change too much in any $n$ consecutive steps, which, after rescaling, will enable an application of Lemma \ref{lem:azuma}.
    First, note that for any $1 \leq j < n$, we have:
    \begin{equation}\label{eq:cover prob}
    \Prob \left[ uv \in G(S_{i+n}) \given H_{i+j-1}, S_{i+j-1}, uv \notin G(S_{i+j}) \right] \leq \frac{2}{n}.
    \end{equation}
    Indeed, the triangles $t_{i+j+1},\ldots,t_{i+n}$ are a uniformly random subset of size $n-j \leq n$, that are chosen from a set of size at least $n^3/2$. Furthermore, the number of triangles containing $uv$ is bounded from above by $n$. Thus, the probability that one of these was chosen is at most $2n^2/n^3 = 2/n$.
    
    Now, let $1 \leq i \leq n^{2+\delta} - n$ and $1 \leq j < n$. By the law of total probability and Inequalities \eqref{eq:prob} and \eqref{eq:cover prob}, for any $H_{i+j},S_{i+j}$ s.t.\ $uv \notin G(S_{i+j})$, it holds that
    \begin{equation}\label{eq:change prob}
    \begin{aligned}
    &\Prob \left[ d_{uv} (i+j) \neq d_{uv} (i+j-1) \given H_{i+j-1},S_{i+j-1}, uv \notin G(S_{i+n}) \right] \\
    & \leq \frac{\Prob \left[ d_{uv} (i+j) \neq d_{uv} (i+j-1) \given H_{i+j-1},S_{i+j-1}, uv \notin G(S_{i+j}) \right]}{\Prob \left[ uv \notin G(S_{i+n}) \given H_{i+j-1}, S_{i+j-1} uv \notin G(S_{i+j}) \right]} \leq \frac{5}{n}.
    \end{aligned}
    \end{equation}
    
    Now, for $T \in \binom{[n]}{\log n}$ let $B_T$ denote the event that for every $j \in T$, $d_{uv}(i+j) \neq d_{uv}(i+j-1)$. Applying a chain of conditional probabilities together with Inequality \eqref{eq:change prob}, for any $T \in \binom{[n]}{\log n}$:
    \[
    \Prob \left[ B_T \given S_i, H_i, uv \notin G(S_{i+n}) \right] \leq \left( \frac{5}{n} \right)^{\log n}.
    \]
    
    We will now use a union bound over all events $B_T$ to show that w.v.h.p.\ in any $n$ consecutive steps of the hypergraph process and for any uncovered edge $uv$, conditioning on the event that $uv$ remains uncovered during these steps, $d_{uv}$ changes by at most $\log n$. Indeed:
	\begin{equation}\label{eq:bounded_d}
	\begin{gathered}
	\Prob \left[ d_{uv} \left( i+n \right) \leq d_{uv} \left( i \right) - \log n \given H_i,S_i, uv \notin G \left(S_{i+n} \right) \right]
    \leq \binom{n}{\log n} \left( \frac{5}{n} \right)^{\log n} \\
    \leq \left( \frac{5en}{n\log n} \right)^{\log n} = n^{-\omegaone} .
	\end{gathered}
	\end{equation}
	
	We treat the evolution of the variables $c_v$ in a similar fashion. At time $i+1$, $c_v$ decreases iff $v \in t_{i+1} \in A_i$, in which case $c_v \left( i+1 \right) = c_v \left( i \right) - 1$. Thus:
	\begin{align}\label{eq:c change}
    \begin{split}
	& \Prob \left[ c_v \left( i+1 \right) \neq c_v \left( i \right) \given H_i, S_i \right] = \frac{1}{ \left( 1 - O \left( n^{\delta - 1} \right) \right) n^3} \frac{1}{2} \sum_{vu \notin G \left(S_i\right) } d_{vu} \left( i \right)  \leq \frac{2}{n} \\
	& \E \left[ c_v \left( i+1 \right) - c_v \left( i \right) \given H_i, S_i \right]
	= - \Prob \left[ c_v \left( i+1 \right) \neq c_v \left( i \right) \given H_i, S_i \right] .
    \end{split}
	\end{align}
	By reasoning similar to that above, for any vertex $v$ and any $0 \leq i \leq n^{2+\delta} - n$:
	\begin{equation}\label{eq:bounded_deg_step}
	\Prob \left[ c_v \left( i+n \right) \leq c_v \left( i \right) - \log n \given H_i, S_i \right] \leq n^{-\omegaone} .
	\end{equation}
	
	At this point it is convenient to rescale our variables. For $0 \leq T \leq n^{1+\delta}$, a vertex $v$, and an edge $uv$ we define:
	\begin{align*}
	\begin{split}
	& C_v \left(T\right) = c_v \left(nT\right) \\
	& D_{uv} \left(T\right) = d_{uv} \left(nT\right) .
	\end{split}
	\end{align*}
	
	\newcommand{\step}{{n^\varepsilon}}
	\newcommand{\alphak}{{\alpha \left(k\right)}}
    \newcommand{\alphat}{{\alpha \left( T \right)}}
	\newcommand{\alphao}{{ n^{ 1 + \delta - \frac{\varepsilon}{3} } = n^{\frac{253}{300}} }}
	\newcommand{\Tn}{{\left(\frac{T}{n}\right)}}
    
    Let $\varepsilon = \frac{1}{2}$. We will prove that for all $T<n^{1+\delta}$ of the form $T = k \step$ (where $k \in \left\{ 0,1,\ldots,n^{2+\delta-\varepsilon} \right\}$), every vertex $v$, and every edge $uv$:
   	\begin{align}\label{eq:trajectory}
	\begin{split}
	& C_v \left(T\right) = n y \Tn \pm \alphat \\
	& D_{uv} \left(T\right) = n z \Tn \pm \alphat .
	\end{split}
	\end{align}
    Where:
    \begin{equation*}
	\begin{split}
	& \alpha \left( 0 \right) = \alphao \\
	& \alpha \left( T + \step \right) = \alphat \left( 1 + \frac{20 \step }{n + 2T}\right) .
	\end{split}
	\end{equation*}
    
    It is straightforward to verify that for all $T$:
	\begin{equation*}
	\begin{split}
	&\alphat \leq \alpha \left( n^{1+\delta} \right) = O \left( n^{1 + 11 \delta - \frac{\varepsilon}{3}} \right) = O \left(n^{\frac{91}{100}}\right) \\
	& n y \left(\frac{T}{n}\right)
	\geq n z \left(\frac{T}{n}\right)
	\geq n z \left(\frac{n^{1+\delta}}{n}\right)
	= \Omega \left( n^{1 - \delta} \right)
	= \Omega \left( n^{\frac{99}{100}} \right) .
	\end{split}
	\end{equation*}
	And so:
	\begin{equation}\label{eq:small_error}
	\begin{split}
	\alphat
	= o \left( nz \Tn \right)
	, o \left( ny \Tn \right) .
	\end{split}
	\end{equation}
    Together, Equalities \eqref{eq:trajectory} and \eqref{eq:small_error} imply the proposition.
	
	We will prove that if \eqref{eq:trajectory} holds for $T$, then w.v.h.p.\ \eqref{eq:trajectory} also holds for $T + \step$. Since $C_v\left(0\right) = D_{uv} \left(0\right) = n$, an inductive argument completes the proof.
	
	Assume \eqref{eq:trajectory} holds for some $T$. Let $uv$ be an edge. Our first order of business is to calculate the expected change in $D_{uv}$ in a single time step. Let $T \leq i < T+\step$. If $uv$ is covered at time $i+1$ then \eqref{eq:trajectory} holds by definition. Therefore we condition on $uv \notin G \left(S_{i+1}\right)$. For compactness, we set $F_i =  \left( H_{i n},S_{i n}, uv \notin G \left( S_{ \left(i+1\right) n} \right) \right)$. Now, by definition:
	\begin{align*}
	\E \left[ D_{uv}\left( i+1 \right) - D_{uv}\left( i \right) \given F_i \right]
	= \sum_{j = 1}^{n} \E \left[ d_{uv} \left( i n + j \right) - d_{uv} \left( i n + j -1 \right) \given F_i \right].
    \end{align*}
    Equality \eqref{eq:expec} holds for any choice of $t_{in+1},\ldots,t_{in + j - 1}$. Furthermore, $uv \notin G \left( S_{ \left(i+1\right) n} \right)$ implies $uv \notin G \left( S_{in+j} \right)$. Therefore:
    \begin{align*}
    \E \left[ d_{uv} \left( i n + j \right) - d_{uv} \left( i n + j -1 \right) \given F_i \right]
    = - \Prob \left[ d_{uv} \left( i n + j \right) \neq d_{uv} \left( i n + j - 1 \right) \given F_i \right].
    \end{align*}
    
    Now, for $j \in [n]$ let $B_j$ be the event that for some edge $ab \notin G \left( S_{ \left(i+1\right) n} \right)$:
    \[
    \left| d_{ab} \left( i n + j - 1 \right) - d_{ab} \left( Tn \right) \right| \geq (i+1-T) \log n.
    \]
    By Inequality \eqref{eq:bounded_d}, $\Prob \left[ B_j \given  F_i \right] = n^{-\omegaone}$. Note that if $\overline{B_j}$ holds, then for all $ab \notin G \left( S_{ \left(i+1\right) n} \right)$, it holds that $d_{ab} \left( i n + j - 1 \right) = d_{ab} \left( Tn \right) \pm \alphat = n z\Tn \pm 2 \alphat$. Therefore, applying the law of total probability:
    \begin{align*}
    \Prob &\left[ d_{uv} \left( i n + j \right) \neq d_{uv} \left( i n + j - 1 \right) \given F_i \right] \\
    &= \Prob \left[ d_{uv} \left( i n + j \right) \neq d_{uv} \left( i n + j - 1 \right) \given F_i, \overline{B_j} \right] \pm n^{-\omegaone}\\
    &= \frac{2 \left( n z \Tn \pm 4 \alphat \right)^2}{\left( 1 - O \left( n^{\delta-1} \right) \right) n^3}
    = \left( 1 \pm O \left( n^{\delta-1} \right) \right) \frac{2 \left( z \Tn \pm 4 \alphat \right)^2}{n}.
    \end{align*}
    Therefore:
    \begin{align*}
    \E &\left[ D_{uv}\left( i+1 \right) - D_{uv}\left( i \right) \given F_i \right]
    = - \left( 1 \pm O \left( n^{\delta-1} \right) \right) 2 \left( z \Tn \pm 4 \alphat \right)^2\\
    & = - 2 z^2 \Tn \pm \frac{18 z \Tn}{n} \alphat
    = z' \Tn \pm \frac{18}{n + 2T} \alphat.
    \end{align*}

    We cannot apply the Azuma-Hoeffding inequality to $D_{uv}$ directly, as the change in a single time step might be as large as $\Omega (n)$, resulting in a meaningless bound. However, as we have already shown, this is unlikely to happen. We will therefore apply the Azuma-Hoeffding inequality to the conditional probability space in which the random variables we are tracking do not change too much in a single time step. Let $B$ be the event that for some $i$, $\left| D_{uv} \left( T \left( i+1 \right) \right) - T \left( i \right) \right| \geq \log n$. By Inequality \eqref{eq:bounded_d} $\Prob \left[ B \given F_i \right] = n^{-\omegaone}$. Therefore, applying the law of total expectation:
    \begin{equation}\label{eq:shifted_expectation}
    \begin{split}
    \E & \left[ D_{uv}\left( i+1 \right) - D_{uv}\left( i \right) \given F_i , \overline{B} \right] = \\
    & \frac{1}{\Prob \left[ \overline{B} \given  F_i \right]} \E \left[ D_{uv}\left( i+1 \right) - D_{uv}\left( i \right) \given F_i \right]
    - \frac{\Prob [B \given F_i]}{\Prob \left[ \overline{B} \given F_i \right]} \E \left[ D_{uv}\left( i+1 \right) - D_{uv}\left( i \right) \given F_i , B \right] \\
    & = z' \Tn \pm \left( \frac{18}{n + 2T} \alphat + n^{-\omegaone} \right)
    = z' \Tn \pm \frac{19}{n + 2T} \alphat .
    \end{split}
    \end{equation}
    
    We will prove the upper bound in Equation \eqref{eq:trajectory}. The proof of the lower bound is similar. To do so we transform $D_{uv}$ into a supermartingale. Define, for $T \leq i \leq T + \step$:
    \[
    D'_{uv} \left( i \right) \coloneqq D_{uv} \left( i \right) - nz \left( \frac{i}{n} \right) - \left( 1 + \frac{19 \left( i-T \right)}{n + 2T} \right) \alphat .
    \]
    Observe that, conditioning on $\overline{B}$:
    \begin{align*}
    \begin{split}
     & \left| D'_{uv} \left( i + 1 \right) - D'_{uv} \left( i \right) \right|
    \\ & \leq \left| D_{uv} \left(i+1\right) - D_{uv} \left(i\right) \right| + n \left| z \left( \frac{i+1}{n} \right) - z \left( \frac{i}{n} \right) \right| + \alphat \frac{19}{n + 2T}  = O \left( \log n \right) .
    \end{split}
    \end{align*}
    We next show that $\E \left[ D'_{uv} \left( i+1 \right) - D'_{uv} \left( i \right) \given H_{i   n}, S_{i   n} , \overline{B}  \right] \leq 0$, i.e., $D'_{uv}$ is a supermartingale. Taking Equation \eqref{eq:shifted_expectation} into account:
    \begin{align*}
    \begin{split}
    \E & \left[ D'_{uv} \left( i+1 \right) - D'_{uv} \left( i \right) \given H_{i   n}, S_{i   n} , \overline{B} \right] \\
    & = \E \left[ D_{uv} \left( i+1 \right) - D_{uv} \left( i \right) \given H_{i   n}, S_{i   n} , \overline{B} \right] - nz \left( \frac{i+1}{n} \right) + nz \left( \frac{i}{n} \right) - \frac{19}{n + 2T} \alphat \\
    & \leq z' \Tn + \frac{19}{n + 2T} \alphat - n \left( z \left( \frac{i+1}{n} \right) - z \left( \frac{i}{n} \right) \right) - \frac{19}{n + 2i} \alpha \left( i \right) \\
    & \leq z' \Tn - n \left( z \left( \frac{i+1}{n} \right) - z \left( \frac{i}{n} \right) \right) .
    \end{split}
    \end{align*}
    
	By the mean value theorem there exists some $s \in [i/n, (i+1)/n]$ s.t.\ $n \left( z \left( \frac{i+1}{n} \right) - z \left( \frac{i}{n} \right) \right) = z'(s)$. Since $z'$ is increasing it holds that $z'(s) \geq z' \left( \Tn \right)$. Therefore:
    \[
    \E \left[ D'_{uv} \left( i+1 \right) - D'_{uv} \left( i \right) \given H_{i   n}, S_{i   n} , \overline{B} \right]
    \leq z'  \Tn - z'(s) \leq 0.
    \]
	
    Finally, we apply the Azuma-Hoeffding inequality (Lemma \ref{lem:azuma}) with respect to the filtration induced by the random variables $\{H_{i   n},S_{i   n}\}_{i=T}^{T+\step}$, conditioned on $\overline{B}$.
    \[
    \Prob \left[ D'_{uv} \left( T + \step \right) \left(i\right) > \frac{\step}{n + 2T} \alphat \given \overline{B} \right] \leq \exp \left( - \Omega \left( \frac{ \left( \frac{\step}{n + 2T} \alphat \right)^2 }{\step \log^2 n} \right) \right) = n^{-\omegaone}.
    \]
    Since $\overline{B}$ holds w.v.h.p. we have, for all edges $uv$, w.v.h.p.:
    \[
    D_{uv} \left( T + \step \right) \leq n z \left( \frac{T + \step}{n} \right) + \alphat + \frac{20 \step}{n + 2T} \alphat = n z \left( \frac{T + \step}{n} \right) + \alpha \left( T + \step \right) .
    \]
	
	We analyze $C_v$ analogously, while omitting calculations very similar to those above. We focus on the most important step: calculating the expected difference. Assume Equality \eqref{eq:trajectory} holds for $T$ and let $T \leq i \leq T + \step$. Let $B_i$ be the event where for some $v$, $\left| C_v \left( i + 1 \right) - C_v \left( i \right) \right| \geq \log n$ or for some $uv \notin G \left( S_{(i+1)n} \right)$, $\left| D_{uv} \left(i+1\right) - D_{uv} \left( i \right) \right| \geq \log n$. Let $\mathcal{B}_i = \cup_{T \leq j \leq i} B_j$. By Inequality \eqref{eq:bounded_deg_step} $\Prob \left[\mathcal{B}_i \given H_{in},S_{in} \right] = n^{-\omegaone}$. If $\overline{\mathcal{B}_i}$ holds, then $C_v \left( i \right) = C_v \left( T \right) \pm \left( i-T \right) \log n$ and $D_{uv} \left( i \right) = D_{uv} \left( T \right) \pm \left( i-T \right) \log n$. For $j$ between $i   n$ and $(i+1)   n$, each $t_j$ is chosen uniformly at random from $n^3\left(1 - O\left(n^{\delta - 1}\right)\right)$ triangles. Thus, by the inductive hypothesis and Equation \eqref{eq:c change}:
	\begin{equation*}
	\begin{split}
	\E & \left[C_v \left( i + 1 \right) - C_v \left( i \right) \given H_{i   n} , S_{i n} \right]
    = \sum_{j=1}^n \E \left[ c_v \left(in + j \right) - c_v \left( in + j - 1 \right) \given H_{i   n} , S_{i n} \right] \\
	& = - n \cdot \frac{\left(n y \left(\frac{T}{n}\right) \pm 2 \alphat  \right) \left( nz \left(\frac{T}{n}\right) \pm 2 \alphat \right)}{n^3\left(1 \pm O\left(n^{ \delta - 1}\right)\right)}
	= y' \Tn \pm \frac{6 \alphat z \Tn }{n}\\
    & = y' \Tn \pm \frac{6}{n+2T} \alphat .
	\end{split}
	\end{equation*}
	As above, we can apply the Azuma-Hoeffding inequality to an appropriate shifted variable to obtain the result.
    
\end{proof}

\section{Asymptotic Enumeration of Latin Rectangles}\label{ap:rectangle count}

In this section we show that for any $\varepsilon>0$ the number of $(1-\varepsilon)n \times n$ Latin rectangles is asymptotically
\[
\left( \left( 1 + \oone \right) \left( \frac{1}{\varepsilon} \right)^{\varepsilon/(1-\varepsilon)} \frac{n}{e^2} \right)^{(1-\varepsilon)n^2}.
\]

Note that a $(1-\varepsilon)n \times n$ Latin rectangle can be viewed as a sequence of $(1-\varepsilon)n$ disjoint $n \times n $ permutation matrices. We count the number of ways to construct such a sequence matrix by matrix. Suppose we have chosen disjoint permutation matrices $P_1,\ldots,P_{i-1}$. Let $A_i$ be the $(0,1)$-matrix of available entries, i.e., $A_i(s,t) = 1$ iff for all $1\leq j < i$, $P_j(s,t)=0$. Then $A_i$ has $k(i) = n-i+1$ ones in each row and column, and $P_i$ can be any permutation matrix supported by $A_i$. By the permanent bounds of Egorychev--Falikman \cite{egorychev1981solution,falikman1981proof} and Br\'egman \cite{bregman1973some}, the permanent of an $n\times n$ $(0,1)$-matrix M with $k$ ones in each row and column satisfies:
\[
\left( \frac{k}{e} \right)^n \leq Per(M) \leq (k!)^{n/k}.
\]
Thus, the number of choices for the whole process is at least:
\begin{align*}
\prod_{k=\varepsilon n}^{n} \left( \frac{k}{e} \right)^n = \left( \frac{1}{e} \right)^{(1-\varepsilon)n^2} \left( \frac{n!}{(\varepsilon n)!} \right)^n
&\geq \left( \frac{1}{e} \right)^{(1-\varepsilon)n^2} \frac{(n/e)^{n^2}}{(\varepsilon n/e)^{\varepsilon n^2}}\\
&= \left( \left( \frac{1}{\varepsilon} \right)^{\varepsilon/(1-\varepsilon)} \frac{n}{e^2} \right)^{(1-\varepsilon)n^2}.
\end{align*}
On the other hand, the number of choices is bounded from above by
\begin{align*}
\prod_{k=\varepsilon n}^{n} (k!)^{n/k} = \prod_{k=\varepsilon n}^{n} \left( \left( 1 + \oone \right) \frac{k}{e} \right)^n
= \left( \left( 1 + \oone \right) \left( \frac{1}{\varepsilon} \right)^{\varepsilon/(1-\varepsilon)} \frac{n}{e^2} \right)^{(1-\varepsilon)n^2},
\end{align*}
as desired.

\bibliography{latin_box}
\bibliographystyle{amsplain}

\end{document}